\documentclass[12pt]{article}

\usepackage{amsmath,amsthm,amssymb,amsbsy,upref,color,graphicx}

\newcommand{\bib}[4]{\bibitem{#1}{\sc#2: }{\it#3. }{#4.}}

\DeclareMathOperator{\argmin}{argmin}

\begin{document}

\def\ds{\displaystyle}
\def\eps{{\varepsilon}}
\def\N{\mathbb{N}}
\def\R{\mathbb{R}}
\def\A{\mathcal{A}}
\def\EE{\mathcal{E}}
\def\F{\mathcal{F}}
\def\HH{\mathcal{H}}
\def\LL{\mathcal{L}}
\def\M{\mathcal{M}}
\def\PP{\mathcal{P}}
\def\RR{\mathcal{R}}
\def\KK{\mathcal{K}}
\newcommand{\be}{\begin{equation}}
\newcommand{\ee}{\end{equation}}
\newcommand{\cp}{\mathop{\rm cap}\nolimits}

\def\symd{{\scriptstyle\Delta}}
\newcommand{\dss}{\displaystyle}
\newcommand{\g}{\gamma}
\newcommand{\Om}{\Omega}
\newcommand{\om}{\omega}
\newcommand{\G}{\Gamma}
\newcommand{\Sg}{\Sigma}
\newcommand{\lb}{\lambda}
\newcommand{\vps}{\varepsilon}
\newcommand{\ep}{\varepsilon}
\newcommand{\vphi}{\varphi}
\newcommand{\wg}{w \gamma}
\newcommand{\wgp}{w \gamma_p}
\newcommand{\ra}{\rightarrow}
\newcommand{\rau}{\rightharpoonup}
\newcommand{\lra}{\longrightarrow}
\newcommand{\Lra}{\Longrightarrow}
\newcommand{\hr}{\hookrightarrow}
\newcommand{\sr}{\stackrel}
\newcommand{\nN}{{n \in \NN}}
\newcommand{\nif}{{n \rightarrow \infty}}
\newcommand{\kif}{{k \rightarrow \infty}}
\newcommand{\sq}{\subseteq}
\newcommand{\sm}{\setminus}
\newcommand{\ov}{\overline}

\newcommand{\MM}{\text{\it MM} }
\newcommand{\GMM}{\text{\it GMM} }
\newcommand{\ls}{|\partial F|}
\newcommand{\forae}{\text{for a.e. }}
\newcommand{\UUU}{\color{red}}
\newcommand{\EEE}{\color{black}}
\newcommand{\haz}{\widehat}
\newcommand{\epsi}{\varepsilon}

\numberwithin{equation}{section}
\theoremstyle{plain}
\newtheorem{theo}{Theorem}[section]
\newtheorem{lemm}[theo]{Lemma}
\newtheorem{coro}[theo]{Corollary}
\newtheorem{prop}[theo]{Proposition}
\newtheorem{defi}[theo]{Definition}
\theoremstyle{remark}
\newtheorem{rema}[theo]{Remark}
\newtheorem{exam}[theo]{Example}

\begin{center}
{\Large Shape flows for spectral optimization problems}
\end{center}

\begin{center}
{\large D.~Bucur, G.~Buttazzo, U.~Stefanelli}
\end{center}

\bigskip

\begin{abstract}
We consider a general formulation of gradient flow evolution for problems whose natural framework is the one of metric spaces. The applications we deal with are concerned with the evolution of {\it capacitary measures} with respect to the $\gamma$-convergence dissipation distance and with the evolution of domains in spectral optimization problems.
\end{abstract}


\bigskip

\section{Introduction}\label{sintr}

Shape optimization problems received a particular attention from the mathematical community in the last years, both for the several applications that require the design of efficient shapes (for instance in Structural Mechanics and Aerospace Engineering) and for the difficult mathematical problems that have to be solved in order to obtain the existence of optimal solutions. In a very general form, shape optimization problems can be written as minimum problems like
\be\label{shopt}
\min\big\{F(\Omega)\ :\ \Omega\in\A\big\}
\ee
where $\A$ is a suitable family of admissible domains and $F$ is a suitable cost function defined on $\A$. Problems of this kind arise in many fields, and we quote the recent books \cite{all02,besi03,bubu05,hen06,hepi05,pir84,sozo92}, where the reader can find all the necessary details and references.

It is well known that the existence of optimal shapes only occurs in very particular situations, where either some quite severe geometrical constraints are imposed to the admissible domains of the class $\A$ (like for instance {\it convexity}), or where the cost functional satisfies suitable monotonicity conditions (as it happens in several {\it spectral optimization problems}). When the existence of optimal shape fails, one has to deal with {\it relaxed} solutions, that belong to a space much larger than the one of classical domains, and describe efficiently the behaviour of minimizing sequences for problem \eqref{shopt}.

In this paper we are interested in problems of the form \eqref{shopt} arising in spectral optimization: the admissible class $\A$ is made of domains of $\R^d$ and the cost functional $F$ is of one of the following types.

\medskip{\it Integral functionals.} Given a right-hand side $f$ we consider the PDE
$$-\Delta u=f\hbox{ in }\Omega,\qquad u\in H^1_0(\Omega)$$
which provides, for every admissible domain $\Omega\subset\R^d$, a unique solution $u_\Omega$ that we assume extended by zero outside of $\Omega$. The cost $F(\Omega)=J(u_\Omega)$ is obtained by taking
$$J(u)=\int_{\R^d}j\big(x,u(x)\big)\,dx$$
for a suitable integrand $j$.

\medskip{\it Spectral functionals.} For every admissible domain $\Omega$ we consider the Dirichlet Laplacian $-\Delta$ which, under mild conditions on $\Omega$, admits a compact resolvent and so a discrete spectrum $\lambda(\Omega)$. The cost is in this case of the form
$$F(\Omega)=\Phi\big(\lambda(\Omega)\big)$$
for a suitable function $\Phi$. For instance, by taking $\Phi(\lambda)=\lambda_k$ we may consider the optimization problem for the $k$-th eigenvalue of $-\Delta$:
$$\min\big\{\lambda_k(\Omega)\ :\ \Omega\in\A\big\}.$$

We will summarize some known facts about the minimization problems above and we deal with the problem of studying the shape evolution $\Omega(t)$, starting from a given domain $\Omega_0$ according to a suitable definition of {\it gradient flow}. 
The theory of gradient flows in metric spaces has been recently developed in a great generality (see \cite{ags05}) and in some situations it can be easily adapted to our purposes, in particular when we deal with relaxed problems. The extra compactness of the latter problems is of great help for proving the existence of a relaxed flow. On the counterpart, the flow is made of relaxed domains (capacitary measures in our case) and not of classical domains; we will show some examples in which, even starting from a very smooth initial domain $\Omega_0$, the gradient flow quits the original admissible class $\A$ to evolve in the class of relaxed shapes (see \cite{bbl08,grki11}).

However, for some particular cases of cost functionals $F$, a different gradient flow can be considered; this will be made in Section \ref{sspec} where we show that a careful use of monotonicity properties of $F$ allows to obtain an evolution path $\Omega(t)$ made of classical domains. Some properties of the path $\Omega(t)$ are studied, and some interesting open problems are pointed out.

Let us close this introduction by mentioning that the idea of considering shape flows is somewhat reminiscent of many classical numerical treatments of \eqref{shopt} where an initial (tentative) shape $\Omega_0$ is iteratively improved towards minimization. A numerical gradient flow perspective has in particular already been considered in \cite{No07,No10}  in connection with some applications to image segmentation, optimal shape design, and surface diffusion.

\section{Preliminary tools}\label{stool}

\subsection{Capacity and quasi-open sets}\label{sscapac}

In the following we use the well-known notion of {\it capacity} for a subset $E$ of $\R^d$:
$$\cp(E)=\inf\Big\{\int_{\R^d}(|\nabla u|^2+u^2)\,dx\ :\ u\in{\cal U}_E\Big\}\,,$$
where ${\cal U}_E$ is the set of all functions $u$ of the Sobolev space $H^1(\R^d)$ such that $u\ge1$ almost everywhere in a neighborhood of $E$. If a property $P(x)$ holds for all $x\in E$ except for the elements of a set $Z\subset E$ with $\cp(Z)=0$, we say that $P(x)$ holds {\it quasi-everywhere} (shortly {\it q.e.}) on $E$, whereas the expression {\it almost everywhere} (shortly {\it a.e.}) refers, as usual, to the Lebesgue measure.

A subset $\Omega$ of $\R^d$ is said to be {\it quasi-open} if for every $\eps>0$ there exists an open subset $\Omega_\eps$ of $\R^d$, such that $\cp(\Omega_\eps\symd\Omega)<\eps$, where $\symd$ denotes the symmetric difference of sets. Actually, in the definition above we can additionally require that $\Omega\subset\Omega_\eps$. Similarly, we define {\it quasi-closed} sets. The class of all quasi-open subsets of a given set $D$ will be denoted by $\A(D)$. In the following we always consider subsets $\Omega$ of a {\it bounded} open set $D\subset\R^d$.

A function $u:\R^d\to\R$ is said to be {\it quasi-continuous} (resp. {\it quasi-lower semicontinuous}) if for every $\eps>0$ there exists a continuous (resp. lower semicontinuous) function $u_\eps:\R^d\to\R$ such that $\cp(\{u\ne u_\eps\})<\eps$. It is well known (see for instance \cite{zie89}) that every function $u\in H^1(\R^d)$ has a quasi-continuous representative $\tilde u$, which is uniquely defined up to a set of capacity zero, and given by
$$\tilde u(x)=\lim_{\eps\to0}\frac{1}{|B(x,\eps)|}\int_{B(x,\eps)}u(y)\,dy\,.$$
In the following we always identify, by an abuse of notation, a Sobolev function $u$ with its quasi-continuous representative $\tilde u$, so that a pointwise condition can be imposed on $u(x)$ for quasi-every $x$. In this way, we have for every subset $E$ of $\R^d$
$$\cp(E)=\min\Big\{\int_{\R^d}(|\nabla u|^2+u^2)\,dx\ :\ u\in H^1(\R^d),\ u\ge1\hbox{ q.e. on }E\Big\}.$$
By the identification above, a set $\Omega\subset\R^d$ is quasi-open if and only if there exists a function $u\in H^1(\R^d)$ such that $\Omega=\{u>0\}$. 

The definition of the Sobolev space $H^1_0(\Omega)$ can be extended for a quasi-open set $\Omega$; it is the space of all functions $u\in H^1(\R^d)$ such that $u=0$ q.e. on $\R^d\setminus\Omega$, with norm
$$\|u\|_{H^1_0(\Omega)}=\|u\|_{H^1(\R^d)}.$$
Most of the well-known properties of Sobolev functions on open sets extend to quasi-open sets. In particular, for every $f\in L^2(D)$ there exists a unique solution of the PDE formally written as
\be\label{edirio}
-\Delta u=f\hbox{ in }\Omega,\qquad u\in H^1_0(\Omega)
\ee
that we consider extended by zero on $D\setminus\Omega$. The precise meaning of the equation above has to be given in the weak form
$$u\in H^1_0(\Omega),\qquad\int_D\nabla u\nabla v\,dx=\int_D fv\,dx\quad\forall v\in H^1_0(\Omega)$$
which turns out to be equivalent to the minimization problem
$$\min\Big\{\int_D\Big(\frac{1}{2}|\nabla v|^2-fv\Big)\,dx\ :\ v\in H^1_0(\Omega)\Big\}.$$
We denote the unique solution $u$ of the problem above by $\RR_\Omega(f)$, which defines in this way the {\it resolvent operator} $\RR_\Omega$.

\subsection{$\gamma$-convergence and $w\gamma$-convergence}\label{ssgamma}

The class $\A(D)$ of all quasi-open subsets of $D$ can be endowed with a convergence structure, called $\gamma$-convergence.

\begin{defi}\label{dgammao}
We say that a sequence of quasi-open sets $(\Omega_n)$ in $\A(D)$ $\gamma$-converges to a quasi-open set $\Omega\in\A(D)$ if for every $f\in L^2(D)$ we have that $\RR_{\Omega_n}(f)$ converge to $\RR_\Omega(f)$ weakly in $H^1_0(D)$.
\end{defi}

The following facts for the $\gamma$-convergence can be shown (see for instance \cite{bubu05}).
\begin{enumerate}
\item In Definition \ref{dgammao} it is equivalent to require the weak $H^1_0(D)$ convergence only for $f=1$. In addition, the $\gamma$-convergence of $\Omega_n$ to $\Omega$ is equivalent to the $\Gamma$-convergence (see \cite{dm93}) of the functionals
$$\int_D|\nabla u|^2\,dx\hbox{ if }u\in H^1_0(\Omega_n),\qquad+\infty\hbox{ otherwise}$$
to the functional
$$\int_D|\nabla u|^2\,dx\hbox{ if }u\in H^1_0(\Omega),\qquad+\infty\hbox{ 
otherwise}$$
with respect to the $L^2(D)$ topology.
\item It can be proven (see \cite{bubu05}) that if $\Omega_n\to\Omega$ in the $\gamma$-convergence, the convergence of the resolvent operators $\RR_{\Omega_n}$ to $\RR_{\Omega}$ is in fact in the $\LL\big(L^2(D)\big)$ operator norm. In particular, the spectrum of $\RR_{\Omega_n}$ converges (componentwise) to the spectrum of $\RR_\Omega$, hence the spectrum of $-\Delta$ on $H^1_0(\Omega_n)$ converges (componentwise) to the spectrum of $-\Delta$ on $H^1_0(\Omega)$.
\item The $\gamma$-convergence is metrizable on $\A(D)$; an equivalent distance to the $\gamma$-convergence is given by
$$d_\gamma(\Omega_1,\Omega_2)=\|\RR_{\Omega_1}(1)-\RR_{\Omega_2}(1)\|_{L^2(D)}.$$
\end{enumerate}

The $\gamma$-convergence is not compact; indeed it is possible to construct a sequence $(\Omega_n)$ of domains such that the corresponding solutions $\RR_{\Omega_n}(1)$ do not converge to a function of the form $\RR_\Omega(1)$. The first example of such a sequence was provided by Cioranescu and Murat in \cite{cimu82} by removing from the set $D$ a periodic array of balls of equal radius $r_n\to0$. If the radius is suitably chosen they proved that the weak $H^1_0(D)$ limit of $\RR_{\Omega_n}(1)$ satisfies the PDE
$$-\Delta u+cu=1\hbox{ in }D,\qquad u\in H^1_0(D)$$
where $c>0$ is a constant, and thus the sequence of domains $(\Omega_n)$ cannot $\gamma$-converge to any domain $\Omega$. This is why, in order to study the behaviour of of minimizing sequences of domains, a {\it relaxation} procedure is needed.

The relaxed form of a Dirichlet problem like \eqref{edirio}, has been obtained by Dal Maso and Mosco in \cite{dmmo87} where it is proven that the compactification of the metric space $\big(\A(D),d_\gamma\big)$ is the set $\M_0(D)$ of all nonnegative regular Borel measures $\mu$ on $D$, possibly $+\infty$ valued, such that
$$\mu(B)=0\hbox{ for every Borel set $B\subset D$ with }\cp(B)=0.$$
Note that the measures $\mu\in\M_0(D)$ are not finite, and may take the value $+\infty$ on large parts of $D$. For instance the measure
$$\infty_{D\setminus\Omega}(E)=\begin{cases}
0&\mbox{if }\cp(E\setminus\Omega)=0,\\
+\infty&\mbox{otherwise}\end{cases}$$
belongs to the class $\M_0(D)$.

Given $\mu\in\M_0(D)$ we consider the space $X_\mu(D)$ of all functions $u\in H^1_0(D)$ such that $\int_D u^2\,d\mu<\infty$, endowed with the Hilbert norm
$$\|u\|_{X_\mu(D)}=\Big(\int_D|\nabla u|^2\,dx+\int_D u^2\,d\mu\Big)^{1/2}.$$
This allows us to consider the relaxed form of a Dirichlet problem, formally written as
$$-\Delta u+\mu u=f\hbox{ in }D,\qquad u\in X_\mu(D)$$
and whose precise meaning is given in the weak form
$$u\in X_\mu(D),\qquad\int_D\nabla u\nabla v\,dx+\int_D uv\,d\mu=\int_Df(x)v\,dx\quad\forall v\in X_\mu(D).$$
By the usual Lax-Milgram method we obtain that, for every $\mu\in\M_0(D)$ and every $f\in L^2(D)$, there exists a unique solution $u=\RR_\mu(f)$ of the equation above, which defines the resolvent operator $\RR_\mu$.

If $\Omega\in\A(D)$ and $\mu=\infty_{D\setminus\Omega}$ then the space $X_\mu(D)$ coincides with the Sobolev space $H^1_0(\Omega)$ and $\RR_\Omega(f)=\RR_\mu(f)$. If $f\ge0$, then by maximum principle the solution $\RR_\mu(f)$ is nonnegative too, and then also $f+\Delta u=\mu u$ is nonnegative. On the other hand, if $f>0$ we can write $\mu=(f+\Delta u)/u$ which gives $\mu$ once $u$ is known; of course we have $\mu=+\infty$ whenever $u=0$. Therefore, working with the class $\M_0(D)$ is in this case equivalent to work with the class of functions $\{u\in H^1_0(D),\ u\ge0,\ \Delta u+f\ge0\}$, which is a closed convex subset of the Sobolev space $H^1_0(D)$.

The $\gamma$-convergence can be extended to the relaxed space $\M_0(D)$: we have $\mu_n\to\mu$ in the $\gamma$-convergence if for every $f\in L^2(D)$ (it is equivalent to require it only for $f=1$) the solutions $\RR_{\mu_n}(f)$ converge to $\RR_\mu(f)$ weakly in $H^1_0(D)$. The main properties of the $\gamma$-convergence on the space $\M_0(D)$ are listed below.

\begin{enumerate}
\item The space $\M_0(D)$ endowed with the $\gamma$-convergence is a compact metric space; an equivalent distance to the $\gamma$-convergence is
$$d_\gamma(\mu_1,\mu_2)=\|\RR_{\mu_1}(1)-\RR_{\mu_2}(1)\|_{L^2(D)}.$$

\item The class $\A(D)$ is included in $\M_0(D)$ via the identification $\Omega\mapsto\infty_{D\setminus\Omega}$ and $\A(D)$ is dense in $\M_0(D)$ for the $\gamma$-convergence. Actually also the class of all smooth domains $\Omega$ is dense in $\M_0(D)$.

\item The measures of the form $a(x)\,dx$ with $a\in L^1(D)$ belong to $\M_0(D)$ and are dense in $\M_0(D)$ for the $\gamma$-convergence. Actually also the class of measures $a(x)\,dx$ with $a$ smooth is dense in $\M_0(D)$.

\item If $\mu_n\to\mu$ for the $\gamma$-convergence, then the spectrum of the compact resolvent operator $\RR_{\mu_n}$ converges to the spectrum of $\RR_\mu$; in other words, the eigenvalues of the Schr\"odinger-like operators $-\Delta+\mu_n$ defined on $X_{\mu_n}(D)$ converge to the corresponding eigenvalues of the operator $-\Delta+\mu$.
\end{enumerate}

Since the $\gamma$-convergence is not compact, in order to treat shape optimization problems it is useful to introduce (see \cite{bubu05}) a convergence much weaker than $\gamma$, that makes the class $\A(D)$ compact. We call {\it weak $\gamma$} this new convergence and we denote it by $w\gamma$.

\begin{defi}\label{dwgamma}
We say that a sequence $(\Omega_n)$ of domains in $\A(D)$ $w\gamma$-converges to a domain $\Omega\in\A(D)$ if the solutions $w_{\Omega_n}=\RR_{\Omega_n}(1)$ converge weakly in $H^1_0(D)$ to a function $w\in H^1_0(D)$ (that we may take quasi-continuous) such that $\Omega=\{w>0\}$.
\end{defi}

We list below the main properties of the $w\gamma$-convergence on the space $\A(D)$; for all the related details we refer the reader to \cite{bubu05}.

\begin{enumerate}
\item We stress the fact that, in general, the function $w$ in Definition \ref{dwgamma} does not coincide with the solution $w_\Omega=\RR_\Omega(1)$; this happens only if $\Omega_n$ $\gamma$-converges to $\Omega$, which in general does not occur, because $\gamma$-convergence is not compact on $\A(D)$.

\item The $w\gamma$-convergence is weaker than the $\gamma$-convergence and, by its definition, it is compact, since the sequence $w_{\Omega_n}=\RR_{\Omega_n}(1)$ is bounded in $H^1_0(D)$ so it always has a subsequence $(\Omega_{n_k})$ weakly converging to some function $w\in H^1_0(D)$.

\item If $f\in L^1(D)$ is a nonnegative function, then the mapping $\Omega\mapsto\int_\Omega f\,dx$ is $w\gamma$-lower semicontinuous on $\A(D)$.

\item If $F:\A(D)\to[-\infty,+\infty]$ is a $\gamma$-lower semicontinuous shape functional which is monotone decreasing with respect to the set inclusion, then $F$ is $w\gamma$-lower semicontinuous. For instance, integral functionals like $\int_Dj(x,u_\Omega)\,dx$ with $j(x,\cdot)$ decreasing, where $u_\Omega=\RR_\Omega(f)$ and $f\ge0$, and spectral functionals like $\Phi\big(\lambda(\Omega)\big)$ with $\Phi$ increasing in each variable, are $w\gamma$-lower semicontinuous.

\end{enumerate}

\subsection{Minimizing movements}\label{ssmovem}

We recall here some notions and results in the direction of variationally-driven evolutions in metric spaces. In particular, we shall mention {\it generalized minimizing movements} and {\it curves of maximal slope} and their relation with {\it gradient flows} whenever a Hilbertian structure is available. In particular, we summarize concepts and results of interest for our purposes, referring to \cite{ags05} for further details.

In all of the following, $(X,d)$ is a complete metric space, $u_0\in X$ is an initial condition, and $F:X\to ]-\infty,+\infty]$ is a proper functional defined on $X$ with {\it effective domain} $D(F)=\{x\in X\ :\ F(x)<+\infty\}$. We let $\tau(d)$ denote the topology in $X$ induced by the metric $d$.

At first, we shall mention the so-called {\it minimizing movements} theory which was introduced by De Giorgi in \cite{dg93} in order to study evolution problems with an underlying variational structure. The framework of the theory is very general and applies both to quasistatic evolutions as well as to gradient flows, under rather mild assumptions.

For every fixed $\eps>0$ the {\it implicit Euler scheme} of time step $\eps$ and initial condition $u_0$ consists in constructing a function $u_\eps(t)=w([t/\eps])$, where $[\cdot]$ stands for the integer part function, in the following way
$$w(0)=u_0,\qquad w(n+1)\in\argmin\left\{F(v)+\frac{d^2(v,w(n))}{2\eps}\right\}.$$

\begin{defi}[Minimizing movements]\label{dmovem}
We say that $u:[0,T]\to X$ is a \emph{minimizing movement} associated to the functional $F$ and the topology $\tau$, with initial condition $u_0$, and we write $u\in\MM(F,\tau,u_0)$, if
\begin{equation}\label{convu}
u_{\eps}(t)\stackrel{\tau}{\to}u(t)\qquad\forall t\in[0,T].
\end{equation}
If the latter convergence holds for a subsequence $\eps_n\to0$, we say that $u:[0,T]\to X$ is a \emph{generalized minimizing movement} and we write $u\in\GMM(F,\tau,u_0)$.

\end{defi}

We say that a trajectory $u:[0,T]\to X$ belongs to $AC^p(0,T;U)$, $p\in[1,\infty]$, if there exists $m\in L^p(0,T)$ such that
\be\label{metric_dev}
d(u(s),u(t))\le\int_s^t m(r)\,dr\qquad\text{for all $0<s\le t<T$.}
\ee
One can prove that, for all $u\in AC^p(0,T;X)$, the limit
$$|u'|(t)=\lim_{s\to t}\frac{d(u(s),u(t))}{|t-s|}$$
exists for a.a. $t\in (0,T)$. This limit is usually referred to as the {\it metric derivative} of $u$ at $t$. In particular, the map $t\mapsto|u'|(t)$ turns out to belong to $L^p(0,T)$ and is minimal within the class of functions $m\in L^p(0,T)$ fulfilling \eqref{metric_dev}, see \cite[Sec. 1.1]{ags05}. Let us recall \cite[Prop.~2.2.3, p.~45]{ags05} the following.

\begin{theo}[Existence of generalized minimizing movements]\label{texmm}
Let the sublevels of $F$ be $\tau$-compact in $X$. Then, for every initial condition $u_0\in D(F)$ the set $\GMM(F,\tau,u_0)$ is non-empty. Moreover, we have that $\GMM(F,\tau,u_0)\subset AC^2(0,T;X)$.
\end{theo}

\subsection{Curves of maximal slope}

We say that a function $g:X\to[0,+\infty]$ is a {\it strong upper gradient} for the functional $F$ if, for every curve $u\in AC(0,T;X)$, the function $g\circ u$ is Borel and \cite[Def.~1.2.1, p. 27]{ags05}
\be\label{e:2.3}
|F(u(t))-F(u(s))|\le\int_s^t g(u(r))|u'|(r)\,dr
\quad\text{for all $0<s\le t<T$}.
\ee
In particular, if $g$ is a strong upper gradient for the functional $F$ and $g\circ u\in L^1(0,T)$ we have that $F\circ u$ is absolutely continuous and
$$|(F\circ u)'|\le (g\circ u)|u'|\qquad\text{a.e. in }(0,T).$$

\begin{defi}[Curve of maximal slope]
Let $g:X\to[0,+\infty]$ be a strong upper gradient for $F$. A trajectory $u\in AC(0,T;X)$ is said to be a \emph{curve of maximal slope} for $F$ with respect to its strong upper gradient $g$ if
\be\label{e:differential-equality}
{}-(F\circ u)'(t)=|u'|^2(t)=g^2(u(t))\qquad\forae t\in(0,T)
\ee
In particular, $F\circ u$ is locally absolutely continuous in $(0,T)$, $g\circ u\in L^{2}(0,T)$, and the energy identity
\begin{gather}\label{abstra-enide}
\frac{1}{2}\int_s^t |u'|^2(r)\,dr+\frac{1}{2}\int_s^t g^2(u(r))\,dr+F(u(t)) =F(u(s))
\end{gather}
holds in each interval $[s,t]\subset(0,T)$.
\end{defi}

The notion of curve of maximal slope is the natural extension to metric spaces of gradient flows in the Hilbertian setting. In particular, in case $X$ is a Hilbert space with scalar product $\langle\cdot,\cdot\rangle$ and norm $\|\cdot\|$ and $F$ is, say, Fr\'echet differentiable one can readily check that 
\begin{align}
&u'+DF(u)=0\ \Longleftrightarrow\ \frac12\|u'+DF(u)\|^2=0\nonumber\\
&\quad\Longleftrightarrow\ \frac12\|u'\|^2
+\frac12\|DF(u)\|^2+\langle DF(u),u'\rangle=0\nonumber\\
&\quad\Longleftrightarrow\ -(F\circ u)'=\|u'\|^2=\|DF(u)\|^2.\label{h}
\end{align} 
Hence, in the case of a smooth functional $F$ the two notions of gradient flow and curve of maximal slope coincide. More generally, curves of maximal slope in a Hilbert space correspond to gradient flows whenever some mild assumption is made on the Fr\'echet subdifferential $\partial F$ of $F$. The latter subdifferential is defined at points $u\in D(F)$ as
$$v\in\partial F(u)\ \Longrightarrow\ \limsup_{w\to u}\frac{F(w)-(F(u)+\langle v,w-u\rangle)}{\|w-u\|}\ge0$$
with $D(\partial F)=\{u\in D(F)\ :\ \partial F(u)\neq\emptyset\}$. In particular, we have the following \cite[Prop. 1.4.1, p. 34]{ags05}.

\begin{prop}[Curves of maximal slope $=$ gradient flows] Let $(X,d)$ be a Hilbert space endowed with its strong topology. Moreover, assume that $\partial F(u)$ is weakly closed for every $u\in D(\partial F)$. Then, $u$ is a curve of maximal slope for $F$ with respect to  $ v\mapsto\|\partial^\circ F(v)\|$  if and only if
$$\begin{cases}
t\mapsto F(u(t))\quad\hbox{is a.e. equal to a non-decreasing function,}\\
u'(t)+\partial^\circ F(u(t))\ni0\quad\hbox{for a.e. }t\in(0,T),
\end{cases}$$
where $\partial^\circ F(u)$ is the subset of elements of minimal norm in $\partial F(u)$.
\end{prop}

In the general metric setting the natural candidate for serving as a strong upper gradient for $F$ is the {\it local slope} (see \cite{ags05,Cheeger99,DeGiorgi80}) of $F$ defined at $u\in D(F)$ as
$$\ls(u)=\limsup_{v\to u}\frac{(F(u)-F(v))^+}{d(u,v)}.$$
Note that indeed the local slope plays the role of the norm of $\partial F$ (see \eqref{h}). In particular, in case $X$ is a Hilbert space and $F$ is Fr\'echet  differentiable  at $u\in D(F)$, then $\ls(u)=\|\partial F(u)\|$.

In general, the function $u\mapsto\ls(u)$ cannot be expected to be lower
semicontinuous. On the other hand, semicontinuity is crucial in order to possibly pass to the limit in \eqref{abstra-enide} (or, rather, in its time-discrete analogue). A way out from this obstruction consists in restricting the analysis to {\it $\lambda$-geodesically convex} functionals. In particular, we say that a trajectory $\gamma:[0,1]\to X$ is a {\it constant-speed geodesic} if
$$d(\gamma(s),\gamma(t))=(t-s)d(\gamma(0),\gamma(1))\qquad\forall0\le s\le t\le T$$
and that a functional $F$ is {\it $\lambda$-geodesically convex} for $\lambda\in\R$ if, for all $u_0,u_1\in D(F)$, there exists a constant-speed geodesic $\gamma$ with $\gamma(0)=u_0$ and $\gamma(1)=u_1$ such that 
$$F(\gamma(t))\le(1-t)F(u_0)+tF(u_1)-\frac{\lambda}{2}t(1-t)d^2(u_0,u_1) \qquad\forall t\in[0,1].$$
In case $X$ has a linear structure, we shall simply (and classically) refer to the latter convexity condition as {\it $\lambda$-convexity}.
 
If $F$ is $\lambda$-geodesically convex and $\tau(d)$-lower semicontinuous then \cite[Cor. 2.4.10, p. 54]{ags05} the local slope $\ls$ is a strong upper gradient for $F$ and it is $\tau(d)$-lower semicontinuous as well.  The same holds if we relax the geodesic convexity assumption above by asking for the weaker property
\begin{align}
&\text{for all $ v_0,\, v_1 \in D(F)$ there exists a curve $\gamma$ with $\gamma(0)=v_0$ and $\gamma(1)=v_1$ such that}\nonumber\\
&v \mapsto \Phi(\epsi,v_0,v):=\frac{1}{2\epsi} d(v,v_0) + F(v) \ \text{is $(\epsi^{-1 }+ \lambda)$-convex on $\gamma$ for all $0<\epsi<1/\lambda^-$}
  \label{ca} 
\end{align}
along with the convention $1/\lambda^-=+\infty$ for $\lambda \geq 0$.

In particular, this entails the following \cite[Thm. 2.3.3, p. 46]{ags05}.

\begin{theo}[$\GMM$ are curves of maximal slope] Let $F$ fulfill the convexity assumption \eqref{ca} being  $\tau(d)$-lower semicontinuous, and coercive, namely
$$\exists  \epsi^*>0,\,u^*\in X:\ \inf  \Phi(\epsi^*,u^*,\cdot) >-\infty.$$
Then, given $u_0\in D(F)$, every $u\in \GMM(F,\tau(d),u_0)$ is a curve of maximal slope for the functional $F$ with respect to its strong upper gradient $\ls$.
\end{theo}

A suitably strengthened version of the convexity assumption \eqref{ca} provides the possibility of proving a generation result. In particular, we shall be dealing with the following
\begin{align}
&\text{for all $ v_*, \, v_0,\, v_1 \in D(F)$ there exists a curve $\gamma$ with $\gamma(0)=v_0$ and $\gamma(1)=v_1$ such that}\nonumber\\
&v \mapsto \Phi(\epsi,v_*,v) \ \text{is $(\epsi^{-1 }+ \lambda)$-convex on $\gamma$ for all $0<\epsi<1/\lambda^-$.}
  \label{ca2} 
\end{align}
Note that property \eqref{ca2} is stronger than the former \eqref{ca} for the latter follows as a particular case with $v_* = v_0$. On the other hand, \eqref{ca2} combines in a crucial way the geodesic convexity properties of the functional and of the curvature properties of the underlying metric space \cite{ags05}. One can in particular check that \eqref{ca2} ensues if, in addition to the $\lambda$-geodesic convexity of $F$ one requires $v \to d^2(v,v_*)$ to be convex for all $v_*\in D(F)$. This is indeed the case of geodesically convex functionals on non-positively curved metric spaces such as Hilbert spaces \cite{mayer98}.
Note that, in the setting of assumption \eqref{ca2},  no compactness of the sublevels of $F$ is needed in order to prove the existence of curves of maximal slope and we have the following \cite[Thm. 4.0.4, p. 77]{ags05}.

\begin{theo}[Generation of the evolution semigroup]\label{gen}
Let $F$  fulfill the convexity assumption \eqref{ca2} being $\tau(d)$-lower semicontinuous and coercive. Then, for any given $u_0\in\overline{D(F)}$ there exists a unique $u=S(u_0)\in \MM(F,\tau(d),u_0)$. Moreover, $u$ is a locally Lipschitz curve of maximal slope for $F$ with respect to its strong upper gradient $\ls$, $u(t)\in D(\ls)$ for all $t\in(0,T)$, and the map $t\mapsto S(u_0)(t)$ is a $\lambda$-contraction semigroup, namely
$$d(S(u_0)(t),S(v_0)(t))\le e^{-\lambda t}d(u_0,v_0)\quad\forall u_0,v_0\in\overline{D(F)}.$$
\end{theo}

\section{Curves of maximal slope of capacitary measures}\label{smeas}

In this section we consider the compact metric space $(\M_0(D),d_\g)$ of capacitary measures endowed with the distance $d_\g$ introduced in Section \ref{ssgamma}. Let $F:\M_0(D)\to]-\infty,+\infty]$ be a $\gamma$-lower semicontinuous (relaxed shape) functional. By Theorem \ref{texmm} for every initial condition $\mu_0\in\M_0(D)$ with $\mu_0\in D(F)$ there exists $\mu\in \GMM(F,\tau(d_\g),\mu_0)$ and the discrete implicit Euler scheme reads
\be\label{bba01}
\mu_\vps^{n+1}\in\argmin\Big\{F(\mu)+\frac{1}{2\vps}d^2_\g(\mu_\vps^n,\mu)\Big\}.
\ee
The main purpose of this section is to study some properties of the generalized minimizing movement $\mu(t)$ and to see when it happens to be a curve of maximal slope.

There is a natural one-to-one map between $\M_0(D)$ and the convex set
\be\label{convexx}
X=\{w\in H^1_0(D)\ :\ w\ge0,\ 1+\Delta w\ge0\}\sq L^2(D),
\ee
given by
$$\mu\mapsto w_\mu:= \RR_\mu(1),\quad\hbox{with inverse}\quad w\mapsto\mu_w:=\frac{1+\Delta w}{w}\;.$$
Moreover, the metric structure on $\M_0(D)$ and $X$ is the same, since 
$$d_\g(\mu_1,\mu_2)=\|w_{\mu_1}-w_{\mu_2}\|_{L^2(D)}.$$
Therefore, every functional $F:\M_0(D)\to]-\infty,+\infty]$ can be identified with a functional $J:L^2(D)\to]-\infty,+\infty]$ with $D(J)\subset X$ by
$$F(\mu)=J(w_\mu)\quad\hbox{or, equivalently,}\quad J(w)=F(\mu_w).$$
The variational flow for $F$ in $\M_0(D)$ can be then obtained through the gradient flow of $J$ in $L^2(D)$, generated by the implicit Euler scheme
\be\label{eimplw}
w_\vps^{n+1}\in\argmin\Big\{J(w)+\frac{1}{2\vps}\int_D|w-w_\vps^n|^2\,dx\Big\}.
\ee

\begin{theo}[Monotonicity]\label{bba03}
Assume that $J:X\to]-\infty,+\infty]$ is decreasing, in the sense
\be\label{bba04}
w_1,w_2\in D(J),\ w_1\le w_2\hbox{ a.e.}\ \Lra\ J(w_1)\ge J(w_2).
\ee
Then, every $w\in \GMM(J,\tau(d_{L^2(D)}),w_{\mu_0})$ is increasing, in the sense that
$$t_1<t_2\ \Lra\ w(t_1)\le w(t_2)\hbox{ a.e.}$$
\end{theo}

\begin{proof}
Let $J$ be monotone in the sense of \eqref{bba04}. Given $w_\vps^n$, in the incremental step \eqref{eimplw} we have to solve the minimum problem
$$\min\Big\{J(w)+\frac{1}{2\vps}\int_D\|w-w_\vps^n\|^2\,dx\Big\}.$$
For every $w\in X$, the function $\max\{w,w_\vps^n\}$ still belongs to $X$,
since the maximum of two subharmonic functions is also subharmonic. Relying on the monotonicity of $J$ we then have that
$$J(\max\{w,w_\vps^n\})
+\frac{1}{2\vps}\int_D\|\max\{w,w_\vps^n\}-w_\vps^n\|^2\,dx
\le J(w)+\frac{1}{2\vps}\int_D\|w-w_\vps^n\|^2\,dx,$$
the inequality being strict as soon as $\max\{w,w_\vps^n\}\ne w$. Consequently, any minimizer $w:=w_\vps^{n+1}$ should satisfy $w\ge w_\vps^n$ a.e., and thus any discrete flow is increasing. Passing to the limit as the step size goes to zero, we obtain that any generalized minimizing movement is increasing.
\end{proof}

\begin{rema}\label{remmon}
We underline that the monotonicity assumption \eqref{bba04} is not equivalent to the monotonicity of measures. If $\mu_1\le\mu_2$ in the classical sense of measures or, weaker, in the sense
\be\label{emonmu}
\int_D\vphi^2(x)d\mu_1\le\int_D\vphi^2(x)\,d\mu_2\qquad\forall\vphi\in H^1_0(D),
\ee
then $w_{\mu_1}\ge w_{\mu_2}$ q.e. The converse is in general false; here is an example:
$$\mu_1=\infty\lfloor_{B(0,1)^c},\qquad\mu_2=1_{B(0,1)}dx+\infty\lfloor_{B(0,R)^c},$$
where $R$ is large enough, such that $w_{\mu_2}\ge w_{\mu_1}$ q.e. Clearly, $\mu_1\not\ge\mu_2$.
\end{rema}

\begin{rema}
A typical functional satisfying the monotonicity assumption is a functional depending on $w$, of the form
$$J(w)=\int_D j(x,w(x))\,dx,$$
where $j:D\times\R\to\R$ is continuous and decreasing in the second variable. In particular, we may take $j(x,w)=-w$ which leads to the energy of the system for the constant force $f\equiv1$.
\end{rema}

\begin{rema}
If $\mu_1\le\mu_2$ in the classical sense of measures, or in the weaker sense \eqref{emonmu}, then it is easy to see that $\lb_k(\mu_1)\le\lb_k(\mu_2)$. We do not know if this is still true under the weaker (see Remark \ref{remmon}) condition that $w_{\mu_1}\ge w_{\mu_2}$ q.e.
\end{rema}

\begin{exam}
An interesting question is the following: if we consider the generalized minimizing movement associated to the energy functional $J(w)=-\int_Dw(x)\,dx$ and start from a quasi-open set, will the flow remain in the family of quasi-open sets? 

As we show in the example below, by considering the topology of $\gamma$-convergence and allowing relaxation in general this does not happen, at least at the discrete level. This kind of phenomenon was numerically observed in the framework of quasi-static debonding membranes \cite{bbl08}, where the evolution takes place in the family of {\it relaxed domains}.

Let $D=B(0,2)$, $\Om_0=B(0,2)\sm\partial B(0,1)\sq\R^2$ and $J(w)=-\int_D w (x)\,dx$. Let us first notice that the mapping
\be\label{bba06}
w\mapsto J_\eps(w)=-\int_D w\,dx+\frac{1}{2\vps}\int_D|w-w_0|^2\,dx
\ee
is strictly convex. We will prove that the solution $w$ minimizing the first incremental step is of relaxed form, independently of the size of $\eps>0$. For this purpose, we will first show that $w$ is radially symmetric. It is not clear that the class $X$ in \eqref{convexx} is stable by Schwartz rearrangement, in spite of the fact that we can use the convexity of the mapping above. Indeed, since $w_0(x)=u_0(|x|)$ is radially symmetric, if $w$ is a solution of the incremental step, then any rotation $w\circ R$ of $w$, is also a solution. Using the strict convexity of \eqref{bba06}, we conclude that for any rotation $R$ the equality $w=w\circ R$ holds, so $w$ is radially symmetric.

Assume now by contradiction that $w$ corresponds to a non-relaxed domain, i.e. to a radially symmetric open set. This means that $w=w_{\Om}$, where $\Om$ is a union of open annuli, centered at the origin. 

Denoting by $A(s,t)$ the annulus $B(0,t)\setminus\overline{B(0,s)}$, with $s<t$, it is easy to see that optimal domains can only be of the form $\Omega_s=B(0,s)\cup A(s,2)$, which provide the corresponding radial solutions
$$w_s(x)=u_s(|x|)=\left\{
\begin{array}{ll}
(s^2-|x|^2)/4&\hbox{if }0\le |x|\le s\\
\ds\frac14 \Bigg(4-|x|^2+(s^2-4)\frac{\log(|x|/2)}{\log(s/2)}\Bigg)&\hbox{if }s\le |x|\le2.
\end{array}
\right.$$
In order to prove that relaxation occurs, it is enough to show that it is more effective to relax on $\partial B(0,1)$ the Dirichlet condition. In particular, given
$$\mu=\eps{\mathcal H}^1\lfloor_{\partial B(0,1)}+\infty\lfloor_{\partial B(0,2)},$$
we have that $F(\mu)$ is lower than any value $F(\Omega_s)$. The corresponding solution $\tilde w$ reads 
$$ \tilde w(x)=\left\{
\begin{array}{ll}
\eps+u_0(|x|)&\hbox{if }0\le|x|\le1\\
\ds\eps\frac{\log(|x|/2)}{\log(1/2)}+u_0(|x|)&\hbox{if }1\le|x|\le2,
\end{array}
\right.$$
and hence relaxation occurs at the first discrete step as soon as we prove that
\be\label{follows}
J_\eps(\tilde w)<J_\eps(w_s)\qquad\forall0<s<2.
\ee
Indeed, by defining $f(r)=\min\{1,\log(r/2)/\log(1/2)\}$, we have that
$$J_\eps(\tilde w)=-2\pi\int_0^2 u_0(r)r\,dr-2\pi\eps\int_0^2 f(r)r\,dr
+\pi\eps\int_0^2 f^2(r)r\,dr.$$
Hence, relation \eqref{follows} is equivalent to 
\be\label{follows2}
\int_0^2 (u_s-u_0)r\,dr-\frac{1}{2\eps}\int_0^2 |u_s-u_0|^2r\,dr
< \eps\int_0^2 fr\,dr-\frac{\eps}{2}\int_0^2 f^2r\,dr.
\ee
One can numerically check that the integral in the left-hand side above is non-positive (and indeed vanishes for $s=1$ only). Hence, relation \eqref{follows2} follows, as we have that 
$$\int_0^2\left(f-\frac{1}{2}f^2\right)r\,dr>0.$$
Note that the latter argument is independent of $\eps$. As such, relaxation is expected to happen instantaneously for any generalized minimizing movement starting from $\Omega_0$.
\end{exam}

We shall assume that $J:L^2(D)\to]-\infty,+\infty]$ is $\lambda$-convex, proper, and lower semicontinuous. For instance, $J$ could be of the form
\be
J(w)=\int_Dj(x,w(x))\,dx + I_X(w),\label{one}
\ee
where $j:D\times\R\to\R$ is a normal $\lambda$-convex integrand and $I_X$ is the indicator function of $X$, namely, $I_X(w)=0$ if $w\in X$ and $I_X=+\infty$ elsewhere. An example in this class is the {\it torsional rigidity} functional given by $j(x,w)=w$. Another example for $J$ is 
$$J(w)=\frac{1}{2}\int_D|\nabla w(x)|^2\,dx+\int_D j(x,w(x))\,dx+I_X(w).$$

\begin{prop}\label{lat}
Let the functionals $F:\M_0(D)\to]-\infty,+\infty]$ and $J:L^2(D)\to]-\infty,+\infty]$ with $D(J)\subset X$ be related by $J(w)=F(\mu_w)$ as above. We have the following
\begin{enumerate}
\item[\rm a)] $F$ is $\lambda$-geodesically convex if and only if $J$ is $\lambda$-convex,
\item[\rm b)] $F$ fulfills \eqref{ca2} if and only if $J$ fulfills \eqref{ca2}.
\end{enumerate}
\end{prop}

\begin{proof}

a). Let $F$ be $\lambda$-geodesically convex. Given  $w_0,w_1\in D(J)$ there exists a constant-speed geodesic $\mu:[0,1]\to\M_0(D)$ such that $w_{\mu(i)}=w_i$, $i=0,1$, and
\be\label{later}
F(\mu(t))\le(1-t)F(\mu(0))+tF(\mu(1))
-\frac{\lambda}{2}t(1-t)d_\g(\mu(0),\mu(1)).
\ee
By defining $w(t)=w_{\mu(t)}$, since
\begin{align*}
&\|w(s)-w(t)\|_{L^2(D)}=d_\g(\mu(s),\mu(t))\\
&\quad=(t-s)d_\g(\mu(0),\mu(1))=(t-s)\|w_0-w_1\|_{L^2(D)}
\qquad\forall0\le s\le t\le1,
\end{align*}
we readily have that $w(t)=(1-t)w_0+tw_1$. By using \eqref{later} we conclude for the $\lambda$-convexity of $J$  as
\begin{align}
J(w(t))=F(\mu(t))&\le(1-t)F(\mu(0))+t F(\mu(1))
-\frac{\lambda}{2}t(1-t)d_\gamma^2(\mu(0), \mu(1))\nonumber\\
&=(1-t)J(w_0)+tJ(w_1)
-\frac{\lambda}{2}t(1-t)\|w_0-w_1\|^2_{L^2(D)} .\label{before}
\end{align}

Assume now $J$ to be $\lambda$-convex, 
Fix $\mu_0,\mu_1\in D(F)$. By defining  $w_i = w_{\mu_i}$ for $i=0,1$ and $w(t)=(1-t)w_0+tw_1$ we get that $\mu(t)=\mu_{w(t)}$ is a constant-speed geodesic joining $\mu_0$ and $\mu_1$ as
\begin{align*}
&d_\g(\mu(s),\mu(t))=\|w(s)-w(t)\|_{L^2(D)}\\
&\quad=(t-s)\|w_0-w_1\|_{L^2(D)}=(t-s)d_\g(\mu_0,\mu_1)
\qquad\forall0\le s\le t\le1.
\end{align*}
Hence, by arguing exactly as in \eqref{before} the $\lambda$-geodesic convexity of $F$ follows. 

b). Let $F$ fulfill \eqref{ca2} and $w_*, \, w_0, \, w_1 \in D(J)$ be given. Define $\mu_* = \mu_{w_*}$, $\mu_0 = \mu_{w_0}$, and $\mu_1 = \mu_{w_1}$, and exploit \eqref{ca2} in order to find the curve $t\mapsto \mu(t)$ (possibly not a geodesic) joining $\mu_0$ and $\mu_1$ such that
\begin{align}
 \frac{1}{2\epsi}d_\g^2(\mu(t),\mu_*) + F(\mu(t)) &\leq  \frac{1-t}{2\epsi}d_\g^2(\mu(0),\mu_*) + (1-t)F(\mu_0)  +\frac{t}{2\epsi}d_\g^2(\mu_1,\mu_*) + tF(\mu_1) \nonumber\\
&- \frac{1 + \epsi \lambda}{2 \epsi} t(1-t) d_\g^2(\mu_0,\mu_1).\label{pop}
\end{align}
By letting $w(t)=w_{\mu(t)}$ we have that $\|w(t) - w_*\| = d_\g(\mu(t),\mu_*)$ and $\|w(t) - w(s)\| = d_\g(\mu(t),\mu(s))$ for all $s, t \in [0,T]$. Hence, relation \eqref{pop} entails that $J$ fulfills \eqref{ca2} as well.

On the contrary assume that $J$ fulfills \eqref{ca2} and that $\mu_*, \, \mu_0, \, \mu_1 \in D(F)$ are given. Define $w_* = w_{\mu_*}$  $w_0 = w_{\mu_0}$  $w_1 = w_{\mu_1}$ and let $t \mapsto w(t)$ be the curve whose existence is ensured by \eqref{ca2}. Then, by letting $\mu(t) = \mu_{w(t)}$ and arguing exactly as above we conclude that $F$ fulfills \eqref{ca2} as well. 
\end{proof}

Proposition \ref{lat} is based of the fact that the geometry of $\M_0$ and $L^2(D)$ coincide. Indeed, $\M_0$ is a non-positively curved metric space. As such $\lambda$-geodesic convexity in $\M_0$ implies the convexity property \eqref{ca2}. We shall give a direct proof of this fact in the following.

\begin{prop}[Geodesic convexity $\Rightarrow$ \eqref{ca2}]
 If $F:\M_0(D)\to]-\infty,+\infty]$ is $\lambda$-geodesically convex then it fulfills the convexity property \eqref{ca2}
\end{prop}

\begin{proof}
Let $\mu_*, \, \mu_0, \, \mu_1 \in D(F)$ be given and define $w_*=w_{\mu_*}$, $w_0=w_{\mu_0}$, $w_1=w_{\mu_1}$, $w(t)=(1-t)w_0 + t w_1$, and $\mu(t)=\mu_{w(t)}$. As  $\mu$ is a constant speed geodesic we have that 
\begin{align}
 \Phi(\epsi,\mu_*,\mu(t)) &= \frac{1}{2 \epsi}d_\g(\mu(t),\mu_*) + F(\mu(t)) = \frac{1}{2 \epsi}\|w(t) - w_*\|^2_{L^2(D)} + F(\mu(t))\nonumber\\
&\leq \frac{1-t}{2 \epsi}\|w_0 - w_*\|^2_{L^2(D)} + \frac{t}{2 \epsi}\|w_1 - w_*\|^2_{L^2(D)}  - \frac{t(1-t)}{2\epsi}\|w_0 - w_1\|^2_{L^2(D)} \nonumber\\
&+(1-t)F(\mu_0) + t F(\mu_1) - \frac{\lambda}{2}t(1-t)d_\g^2(\mu_0,\mu_1)\nonumber\\
&=(1-t)\Phi(\epsi,\mu_*,\mu_0)+ t \Phi(\epsi,\mu_*,\mu_1) - \frac{1 + \epsi \lambda}{2 \epsi}t (1-t) d_\g^2(\mu_0,\mu_1)\nonumber
\end{align}
whence the assertion follows.
\end{proof}

We shall now come to the existence results for evolution. Again, this can be formulated equivalently for trajectories of capacitary measures $t\mapsto\mu(t)\in\M_0(D)$ or of their function representatives   $t\mapsto w(t)\in X$. Let us start from measures. Theorem \ref{gen} yields the following.

\begin{theo}[Curves of maximal slope of capacitary measures]
Let $F:\M_0(D)\to]-\infty,+\infty]$ be proper, $d_\g$-lower semicontinuous and $\lambda$-geodesically convex. Then, for any given $\mu_0\in\overline{D(F)}$ there exists a unique $\mu =S(\mu_0)\in \MM(F,\tau(d_\g),\mu_0)$. Moreover, $\mu$ is locally Lipschitz curve of maximal slope for $F$ with respect to $\ls$, $u(t)\in D(|\partial F|)$ for all $t\in(0,T)$, and 
$$d_\g(S(\mu_0)(t),S(\nu_0)(t))\le d_\g(\mu_0,\nu_0)e^{-\lambda t}\quad\forall\mu_0,\nu_0\in\overline{D(F)}$$
\end{theo}

As for the  function  representatives $w$, the situation is that of classical gradient flows in Hilbert spaces.

\begin{theo}[Gradient flow of subharmonic representatives]
Let $J:L^2(D)\to]-\infty,+\infty]$ with $D(J)\subset X$ be proper, lower semicontinuous and $\lambda$-convex. Then, for any given $w_0\in\overline{D(J)}$ there exists a unique $w=S(w_0)\in\MM(F,\tau(d_{L^2(D)}),w_0)$. Moreover, $w$ is the unique solution of gradient flow of the functional $J$ in $L^2(D)$. In particular,
\begin{align}
w'+\partial J(w)\ni0\quad\text{a.e. in }(0,T),\quad w(0)=w_0.\label{eqw}
\end{align}
Eventually, we have that
$$\|S(w_0)(t)-S(v_0)(t)\|\le e^{-\lambda t}\|w_0-v_0\|\quad\forall w_0,v_0\in\overline{D(J)}.$$
\end{theo}

Note that, in case $J$ is defined as in \eqref{one} via a smooth $j$, the inclusion in \eqref{eqw} reads
$$w'+\partial_w j(x,w)+\partial I_X(w)\ni0$$
which is equivalent to $w(t)\in X$ and
$$\int_D\big(w'(t)+\partial_w j(\cdot,w(t))\big)(w(t)-\tilde w)\,dx
\le0\qquad\text{a.e. in }(0,T),\ \forall\tilde w\in X.$$

More generally, in case $J=M+I_X$ where $M:L^2(D)\to]-\infty,+\infty]$ is a proper, convex, and lower semicontinuous functional with $\text{\rm int}\,D(M)\cap X\neq\emptyset$, the inclusion in \eqref{eqw} means $w(t)\in X\cap D(M)$ and
$$\int_D w'(t)(w(t)-\tilde w)\,dx\le M(\tilde w)-M(w(t))\qquad\text{a.e. in }(0,T),\ \forall\tilde w\in X\cap D(M).$$

\section{Variational flows of shapes: spectral optimization problems}\label{sspec}
\subsection{General shape evolution}

In this section, we deal with flows of shapes. There are no ``standard'' distances on the class $\A(D)$ of quasi-open subsets of $D$, and several choices will be studied in the sequel. A first natural distance is given by the Lebesgue measure of the symmetric difference set
$$d_{char}(\Om_1,\Om_2)=|\Om_1\symd\Om_2|.$$
Since two quasi-open sets may differ for a negligible set (think for instance in $\R^2$ to a disk and a disk minus a segment), this is not a proper metric in $\A(D)$, so that one should consider equivalence classes in the family of shapes.

For this purpose, for every measurable set $M\sq D$, we define the Sobolev space
$$\tilde H^1_0(M):=\big\{u\in H^1_0(D),\ u =0\hbox{ a.e. on }D\sm M\big\}.$$
We notice that, for a given open set $\Om$, the space $\tilde H^1_0(\Om)$ may not coincide with the usual Sobolev space $H^1_0(\Om)$, the latter being possibly smaller, as soon as $\Om$ is non smooth. Nevertheless, for every measurable set $M\sq D$ there exists a unique (up to a zero capacity set) quasi-open set $\Om(M)\in\A(D)$ such that
$$\tilde H^1_0(M)=H^1_0\big(\Om(M)\big).$$
If $M_1\sq M_2$ then $\Om(M_1)\sq\Om(M_2)$. In particular, $\Om(M)\sq M$. Consequently, one can define the resolvent of the Laplace operator with Dirichlet boundary conditions associated to $M$ by setting
$$\RR_M:L^2(D)\to L^2(D),\qquad\RR_M:=\RR_{\Om(M)},$$
and one can extend the $\g$ distance and the $\wg$-convergences to measurable sets, by setting
$$d_\g(M_1,M_2)=d_\g\big(\Om(M_1),\Om(M_2)\big),\qquad M_n\sr{\wg}{\lra}M\mbox{ if }\Om(M_n)\sr{\wg}{\lra}\Om(M).$$

Working with measurable sets instead of quasi-open sets in shape optimization problems associated to Sobolev spaces may, in general, severely change the result. Nevertheless, as soon as the functional $F$ satisfies some monotonicity assumption, the problems become, in a certain sense, equivalent. We refer the reader to \cite{bbh09} for more details between this equivalence.

So let us denote
$$M(D)=\{M\sq D\ :\ M\mbox{ measurable}\}.$$
Let $F:\A(D)\to]-\infty,+\infty]$ be a $\wg$-lower semicontinuous functional, monotone decreasing for set inclusion and consider its extension to a functional on measurable sets given by $\haz F:M(D)\to]-\infty,+\infty]$ 
$$\haz F(M)=F(\Omega(M)).$$ 
Functionals of the form
$$F(\Omega)=\Phi\big(\lambda_1(\Omega),\dots ,\lambda_k(\Omega)\big),$$
where $\Phi:\R^k\to]-\infty,+\infty]$ is increasing in each variable and lower semicontinuous are admissible, since for every $k\in\N$ the $k$-th eigenvalue of the Dirichlet Laplacian $\lb_k(\Omega)$ is decreasing with respect to the inclusion of measurable sets. As a consequence, $\haz F$ is monotone decreasing with respect to set inclusion as well.

In $M(D)$ the distance $d_{char}$ is not compact. Nevertheless, relying on the monotonicity assumption on $\haz F$, we have the following result. 

\begin{theo}[Generalized minimizing movements of shapes]\label{bba07}
Let $F:\A(D)\to]-\infty,+\infty]$ be $\wg$-lower semicontinuous, monotone decreasing for set inclusion, and let $M_0\in M(D)$ such that $\Omega(M_0)\in D(F)$. Then, the set $\GMM(\haz F,\tau_{w\gamma},M_0)$ is non-empty. Moreover, every $M\in\GMM(\haz F,\tau_{w\gamma},M_0)$ is increasing in the sense of set inclusion.
\end{theo}

\begin{proof}
We prove first that every solution of the incremental Euler scheme
\be\label{bba08}
M_\vps^{n+1}\in\dss\argmin_{M\in M(D)}\Big\{\haz F(M)+\frac{1}{2\vps}|M_\vps^n\symd M|^2\Big\}.
\ee
is increasing in the sense of inclusions. As a consequence of the monotonicity of $\haz F$, we have that
$$\haz F(M\cup M_\vps^n)+\frac{1}{2\vps}|M_\vps^n\symd(M\cup M_\vps^n)|^2\le \haz F(M)+\frac{1}{2\vps}|M_\vps^n\symd M|^2,$$
the inequality being strict as soon as $|M_\vps^n\sm M|$ is not zero. Consequently, every solution $M$ of the incremental problem satisfies, if it exists, $M_\vps^n\sq M$.
 
Let $\om_k=\Omega(M_k)\sq M_k$ where $(M_k)_k$ is a minimizing sequence of measurable sets, each one containing $M_\vps^n$. By the compactness of the $w\gamma$-convergence, up to extracting a subsequence we have that $\om_k\sr{\wg}{\to}\om$.
As $F(\om)\le\liminf_\kif F(\om_k)$, the measurable set $\om\cup M_\vps^n$ is a solution to the incremental problem \eqref{bba08}, since
\begin{align}
&\haz F(\omega\cup M_k)+\frac{1}{2\eps}|(\omega\cup M_k)\symd M^n_\eps|^2
\le\haz F(\omega)+\frac{1}{2\eps}|(\omega\cup M_k)\symd M^n_\eps|^2\nonumber\\
&=F(\omega)+\frac{1}{2\eps}|(\omega\cup M_k)\symd M^n_\eps|^2\le\liminf_{k\to+\infty}\left(F(\omega_k)+\frac{1}{2\eps}|(\omega_k\cup M_k)\symd M^n_\eps|^2\right)\nonumber\\
&\le\liminf_{k\to+\infty}\left(\haz F(M_k)+\frac{1}{2\eps}|M_k\symd M^n_\eps|^2\right).\nonumber
\end{align}
Finally, we set $M_\vps^{n+1}:=\om\cup M_\vps^n$.

We rewrite the discrete flows in terms of quasi-opens piecewise constant sets $t\mapsto \omega_\epsi(t)$ and we pass to the limit as $\vps\to0$. We reproduce in this setting the argument of \cite[Thm. 3.2]{MM05}. In particular, by using the monotonicity of the flows we have that the functions $\delta_\epsi(t)= |\omega_\epsi(t) \symd \omega_0|$ are non-decreasing. Hence, by the classical Helly principle, at least for some not relabeled subsequence we have that $\delta_\epsi(t) \to \delta(t)$ for all $t \in [0,+\infty)$ and some non-decreasing function $\delta$. The function $\delta$ is continuous with the exception of at most a countable set of points $N$. We shall introduce the countable set $M\subset [0,+\infty)$ in such a way that 
$$0 \in M, \quad M \ \ \text{is dense in}  \ \ [0,+\infty), \quad N \subset M.$$
By a diagonal estraction argument and the compactness of the $\wg$-topology (still not relabeling) one can find that $\omega_\epsi(t) \sr{\wg}{\to}\om(t)$ for all $t \in M$. Let us now fix $t \in  [0,+\infty)\setminus M$, let $t_n \in M$ such that $t_n \nearrow t$, and define $\omega(t) = \cup_{n}\omega(t_n)$. We shall show that indeed $\omega(t)$ coincides with the $\wg$-limit of $\omega_\epsi(t)$. To this aim we exploit again the compactness of the $\wg$-topology, and extract a further (possibly $t$ dependent) subsequence $\epsi^t_n$ in such a way that $\omega_{\epsi^t_n}(t) \sr{\wg}{\to} \omega_*$. By using the lower semicontinuity of the Lebesgue measure with respect to the $\wg$-topology \cite[Prop. 5.6.3, p. 125]{bubu05} We have that 
\begin{align}
 |\omega(t)\symd \omega_*| &= \lim_{n\to + \infty}|\omega(t_n) \symd \omega_*| \leq \lim_{n\to + \infty} \liminf_{k\to + \infty} |\omega_{\epsi^t_k}(t_n)\symd \omega_{\epsi^t_k}(t)| \nonumber\\
&= \lim_{n\to + \infty} \liminf_{k\to + \infty} \big(\delta_{\epsi^t_k}(t) -  \delta_{\epsi^t_k}(t_n)\big) =  \lim_{n\to + \infty} \big(\delta(t) -  \delta(t_n)\big) \sr{t \not \in M}{=}0.\nonumber
\end{align}
Hence $\omega(t) \equiv \omega_*$. In particular, the whole sequence $\omega_\epsi (t)$ $\wg$-converges to $\omega(t)$ even for $t \not \in M$. Finally, we have proved that $ t\mapsto\om(t)$ belongs to $\GMM(F,\tau_{w\gamma},\Omega(M_0))$. Correspondingly, $M_\eps$ has a pointwise limit $M$. In particular, $M$ belongs to $\GMM(\haz F,\tau_{w\gamma},M_0)$.
\end{proof}

\begin{exam}
{\bf Evolution of a ball.} Let $\Om_0=B(0,R_0)$. For every $\vps >0$, the discrete movement associated to $\lb_1$ consists on balls. This is a consequence of the Schwartz rearrangement procedure. Consequently, the minimizing movement consists of an increasing evolution of concentric balls. 

For every given $R>0$ consider the ball $B(0,r)$, with $r>R$, which minimizes
$$\lambda_1(B(0,r))+\frac{\omega_d^2(r^d-R^d)^2}{2\eps}=r^{-2}\lambda_1(B(0,1))+\frac{\omega_d^2(r^d-R^d)^2}{2\eps},$$
where $\omega_d$ denotes the Lebesgue measure of the unit ball in $\R^d$. We obtain
$$-\frac{2}{r^3}\lambda_1(B(0,1))+\frac{\omega_d^2}{\eps}(r^d-R^d)dr^{d-1}=0$$
which gives, for $\eps$ small,
$$r\approx R+\frac{2\lambda_1(B(0,1))}{d^2\omega_d^2R^{2d+1}}\eps.$$
The radius $R(t)$ during the evolution then satisfies the differential equation
$$R'(t)=\frac{2\lambda_1(B(0,1))}{d^2\omega_d^2 R^{2d+1}}$$
which has the solution
$$R(t)=\left(R_0^{2d+2}+\frac{4(d+1)\lambda_1(B(0,1))}{d^2\omega_d^2}t\right)^{1/(2d+2)}.$$
\end{exam}

\begin{rema} {\bf $F$ might be discontinuous.}
The functional $F$ may be discontinuous on the curve $t\mapsto\om(t)$. Indeed, let us consider $F(\Om)=\lb_1(\Om)$ and a generalized minimizing movement starting from $\Om_0=B(0,1)\cup A(R_1,R_2)$ where $A(R_1,R_2)$ stands for the annulus centered at $0$ of radii $1<R_1<R_2$. We choose $R_1$ and $R_2$ in such a way that $\lb_1(A(R_1, R_2))> B(0,1)$. Hence, as the connected component $A(R_1, R_2)$ of $\Omega_0$ does not contribute to the value $F(\Omega_0)$ of the functional,  the intuition hint is that   the generalized minimizing movement from $\Omega_0$ will be $\Omega(t)= A(R_1, R_2) \cup B(0,f(t))$ with $f$ increasing and discontinuous at $f=R_1$.

\end{rema}

An interesting question is whether the evolution is stable in some particular classes of shapes. In particular, the interest in stability is related to compactness. In two dimensions of space, the class of simply connected open sets is compact with respect to $\g$-convergence. Moreover, in any dimension of the space, the class of convex sets is also compact with respect to $\g$-convergence. In the general case, we shall however remark that stability is not to be expected, as we argue below.

\begin{rema} {\bf Topological genus is not conserved.} Let $\Om_0 = A(R_1,R_2)\setminus C$ where $C$ is a radial {\it cut}. As $\lb_1(\Om_0) > \lb_1(A(R_1,R_2))$, any generalized minimizing movement starting from $\Om_0$ will immediately fill-in the cut so that the simply connected $\Om_0$ gets to be non-simply connected.

An example of an evolution from two simply-connected components to one non-simply-connected component is that starting from $\Om_0=B(0,1)\cup U$, where $U$ is a suitable set, disjoint from $B(0,1)$. Let $s=\inf\{r>1\ :\ B(0,r) \cap  U\ne\emptyset\}$ and assume that $\lb_1(U)>\lb_1(B(0,s))$. Since $U$ does not contribute to $\lb_1(B(0,r)\cup U)$ up to $r=s$,   intuitively   any generalized minimizing movement starting from $\Om_0$ is of the form $B(0,f(t))\cup U$ with $f$ increasing up to some intersection time. For suitably chosen sets $U$, after the intersection time the new set will not be simply connected.

Note nonetheless that, due to the monotonicity of the flow, the number of connected components is non-increasing during the evolution.
\end{rema}


\begin{rema}{\bf Convex shapes are unstable.}
Assume $\Omega_0$ to be the square $[0,\pi]^2\subset \R^2$. Consider the mapping $T_t : x \in \R^2 \mapsto x + t v(x) n $ where $v$ is suitably smooth and $n$ is the outward unit normal to $\partial \Omega_0$ and consider $\Omega_t:= T_t (\Omega_0)$.  For the sake of definiteness, we shall normalize $\int_{\partial \Omega_0} |v|ds =1$. We have that \cite[Thm. 5.7.1, p. 209]{hepi05} 
$$\frac{d}{dt} \lb_1(\Omega_t) = - \int_{\partial \Omega_0} \left| \frac{\partial u_1}{\partial n}\right|^2 v \, ds$$
where $u_1$ is the first eigenfunction (with unit $L^2$ norm) of the Dirichlet Laplacian on $\Omega_0$. Hence, by fixing a time step $\epsi>0$ we can readily compute that the minimum of 
$$ t \mapsto G_v(t):=F(\Omega_t) +\frac{1}{2\epsi}|\Omega_t \symd \Omega_0|^2= {}- t \int_{\partial \Omega_0} \left| \frac{\partial u_1}{\partial n}\right|^2 v \,ds + \lambda_1(\Omega_0) + \frac{t^2}{2\epsi} \left(\int_{\partial \Omega_0}|v|ds\right)^2 $$
is attained at 
$$ t_v = \epsi \int_{\partial \Omega_0} \left| \frac{\partial u_1}{\partial n}\right|^2 v\, ds $$
and corresponds to the value
$$G_v(t_v)= -\frac{\epsi}{2}\left(\int_{\partial \Omega_0} \left| \frac{\partial u_1}{\partial n}\right|^2 v \, ds\right)^2  + \lambda_1(\Omega_0). $$
In particular, as we readily compute that  $u_1(x,y)=\sin x \sin y$, the latter entails that in order to minimize $v\mapsto G_v(t_v)$ one would rather have $v \geq 0$ and  concentrate the mass of $v$ in the middle of the sides of the square. Hence, we conclude that bumps are likely to develop from midpoints of the sides of the square so that convexity will be lost.
\end{rema}

We can formulate some open questions:
\begin{itemize}
\item Assume $\Om_0$ is convex. Under which conditions on $\Omega_0$ and $m$ the minimizers of
\be\label{minxx}
\min\big\{\lambda_1(\Omega)\ :\ \Om_0\sq\Om,\ |\Omega|=m\big\}
\ee
are convex? According to the  intuitive   argument above, if $\Omega_0$ is a square and $m$ is slightly larger than $|\Omega_0|$, then the optimal domains should not be convex. In \cite{bbv11} it is proved that for a ``thin'' rectangle $\Omega_0$ of sizes $\eps$ and 1, and $m<\pi/4$, the solution of the shape optimization problem eqref{minxx} cannot be convex, provided $\eps$ is small enough.

\item Let $\Omega_0$ be a convex set and assume that for every $m\ge|\Omega_0|$ every solution of the shape optimization problem \eqref{minxx} is convex. Is it true that then $\Omega_0$ is a ball?

\item Is it true that the generalized minimizing movement associated to $\lb_1$ in the framework of Theorem \ref{bba07} will converge to a ball (rescaling if necessary)?
\item Prove or disprove that the metric derivative of $\lb_1$ computed at a bounded smooth set $\Om$ is given by
$$|\lb_1'|(\Omega) = \max_{\partial \Omega} \left| \frac{\partial u_1}{\partial n}\right|^2.$$
Precisely, prove that
$$\limsup_{d_{char}(\Om_n,\Om)\to0,\ \Om\sq\Om_n}\frac{\lb_1(\Om)-\lb_1(\Om_n)}{|\Om_n\sm\Om|}\le\max_{\partial\Om_0}\Big|\frac{\partial u_1}{\partial n}\Big|^2.$$
\end{itemize}

\noindent{\bf Constraint on the measure.} An alternative evolution, which does not require any rescaling, is to work in the class of sets with prescribed measure. Let $c>0$. We consider only measurable sets $M\in M(D)$ such that $|M|=c$. The incremental problem is given by:
\be\label{bba10}
M_\vps^{n+1}\in\dss\argmin_{M\in M(D),|M|=c}\Big\{\haz F(M)+\frac{1}{2\vps}|M_\vps^n\symd M|^2\Big\},
\ee
No monotonicity can occur in this case, unless the flow is constant. The existence of a generalized minimizing movement associated to the incremental step \eqref{bba10} is not clear. Nevertheless, one can construct discrete solutions of the incremental scheme. 

Indeed, let $(M_k)_k$ be a minimizing sequence in \eqref{bba10}. We associate the quasi open sets $\om_k=\Omega(M_k)$ and, up to a subsequence, we can assume that $\om_n\sr{\wg}{\lra}\om$. If $|\om|=c$, then $1_{M_k}$ converges in $L^1(D)$ to $1_\om$, so that $\om$ is a minimizer.

If $|\om|<c$, we replace $\om$ with $\om\cup U$, where $U$ is chosen such that $|\om\cup U|=c$ and
$$|M_\vps^n\symd M_k|\to|M_\vps^n\Delta(\om\cup U)|.$$
Consequently, $\om\cup U$ is a solution to the incremental step \eqref{bba10}.

An alternative way is to replace the measure constraint by adding a penalized term in the functional, i.e. to replace $\haz F(M)$ by $\haz F(M)+|M|$. In this case, the existence of a solution to the incremental step relies on the lower semicontinuity of the Lebesgue measure for the $\wg$-convergence.

\medskip
\noindent {\bf Perimeter penalization.} One can alternatively introduce a penalization on the perimeter. In this case, the incremental step reads
\be\label{bba11}
M_\vps^{n+1}\in\dss\argmin_{M\in M(D)}\Big\{\haz F(M)+P_D(M)+\frac{1}{2\vps}|M_\vps^n\symd M|^2\Big\},
\ee
The topology given by $d_{char}$ turns out to be compact on the sublevels of $\haz F + P_D$.

\medskip
\noindent{\bf Hausdorff distance.} There are several other geometric distances in the family of open sets, but they are hardly compatible with the $\g$-convergence. Nevertheless some partial observations can be done.

Let $d_{H^c}$ denote the Hausdorff complementary distance in the family of open subsets of $D$, given by
$$d_{H^c}(\Om_1,\Om_2)=\max_{x\in\ov D}|d(x,\ov D\sm\Om_1)-d(x,\ov D\sm\Om_2)|.$$
Assume $F$ is increasing with respect to the set inclusion. Then, there exists a solution of the iteration step
$$\min_{\Om\sq D}F(\Om)+\frac{d_{H^c}^2(\Om,\Om_0)}{2\vps},$$
which is of the form
$$\Om=D\sm(\Om_0^c+\ov B_h).$$
Let $\Omega\subset D$ be open. Then, $\Omega\cap\Omega_0$ is open and, as $F$ is increasing with respect to set inclusion and $d_{H^c}(\Omega\cap \Omega_0,\Omega_0)\le d_{H^c}(\Omega,\Omega_0)$, we have
$$F(\Omega\cap\Omega_0)+\frac{1}{2\eps}d_{H^c}^2(\Omega\cap\Omega_0,\Omega_0)\le F(\Omega ) + \frac{1}{2\eps}d_{H^c}^2(\Omega,\Omega_0).$$
Now, define $h=d_{H^c}^2(\Omega\cap\Omega_0,\Omega_0)$ and observe that $D\setminus(\Omega_0^c+\overline B_h)\sq\Omega\cap\Omega_0$ and $d_{H^c}(D\setminus(\Omega_0^c+\overline B_h),\Omega_0)=h$. Then,
\begin{align}
&F(D\setminus(\Omega_0^c+\overline B_h))+\frac{1}{2\eps}d_{H^c}^2(D\setminus (\Omega_0^c+\overline B_h),\Omega_0)\nonumber\\
&\le F(\Omega\cap\Omega_0)
+\frac{1}{2\eps}d_{H^c}^2(\Omega\cap\Omega_0,\Omega_0)
\le F(\Omega)+\frac{1}{2\eps}d_{H^c}^2(\Omega,\Omega_0)
\quad\forall\Omega\sq D\ \text{open}.\nonumber
\end{align}

Finally, there exists a generalized minimizing movement associated to $F$ and $d_{H^c}$ of the form $t\mapsto D\sm(\Om_0^c+\ov B_{f(t)})$, where $f$ is continuous and increasing. 

The same argument for decreasing functionals $F$ and the Hausdorff distance 
$$d_H(F_1,F_2)=\max_{x\in\ov D}|d(x,F_1)-d(x,F_2)|$$
can be repeated. The only point which is more delicate is concerned with the fact that the Hausdorff distance is not a ``proper'' metric in the family of open sets. Nevertheless, one can prove the existence of generalized minimizing movement associated to $F$ and $d_{H}$ of the form $t\mapsto\mbox{int}(\ov\Om_0+\ov B_{f(t)})$, where $F$ is continuous and increasing. This solution relies on the equivalence relation in the family of the open sets: $\Om_1\equiv\Om_2$ if $\ov\Om_1=\ov\Om_2$ and on the redefinition of the Sobolev space
$$\tilde H^1_0(\Om):=\big\{u\in H^1_0(D)\ :\ u=0\mbox{ a.e. on }D\sm\ov\Om\big\}.$$

\subsection{Flows of convex shapes}

In this section we deal with the evolution of convex open sets. We introduce the family
$$\KK(D)=\{K\sq D\ :\ K\mbox{ open and convex}\}.$$
There are different possible distances on $K$ which have the same convergent sequences
\begin{itemize}
\item the Hausdorff distance;
\item $d_2(K_1,K_2)=\|b_{K_1}-b_{K_2}\|_{L^2(D)}$, where $b_K$ is the oriented distance function, defined by $b_K(x)=-d(x,\partial K)$ for $x\in K$ and $b_K(x)=d(x,\partial K)$ for $x\in D\sm K$;
\item the $L^1$ distance of the characteristic functions $d_{char}(K_1,K_2)=\int_D|1_{K_1}-1_{K_2}|\,dx$.
\end{itemize}

A slightly different distance, defined on the equivalence classes of homotopic convex sets is the Fraenkel relative asymmetry, defined by
$$A(K_1,K_2):=\inf_{x_0\in\R^n}\Big\{\frac{|K_1\symd (x_0+\lambda K_2)|}{|K_1|}\Big\},\quad\mbox{where }\lambda:=\frac{|K_1|^{1/n}}{|K_2|^{1/n}}.$$

Assume that $F:\KK(D)\to\R$ is a $\g$-lower semicontinuous shape functional which satisfies $F(K_n)\to+\infty$ as soon as $K_n$ converges to a degenerate set. Since in the class of convex sets, the $w\gamma$-convergence coincides with the $\gamma$-convergence, it is useless to require $\wg$-lower semicontinuity. Notice that all previous topologies are compact on sublevels of $F$. By applying Theorem \ref{texmm} we have the following.

\begin{theo}[Generalized minimizing movements of convex shapes]
For $d=d_H,d_{char}$, or $d_2$, and for every initial convex set $K_0\in D(F)$, we have that $\GMM(F,\tau(d),K_0)$ is non-empty. 
\end{theo}

\section*{Acknowledgments}
The work of Dorin Bucur is part of the project ANR-09-BLAN-0037 {\it Geometric analysis of optimal shapes (GAOS)} financed by the French Agence Nationale de la Recherche (ANR). The work of Giuseppe Buttazzo and Ulisse Stefanelli is part of the project 2008K7Z249 {\it Trasporto ottimo di massa, disuguaglianze geometriche e funzionali e applicazioni} financed by the Italian Ministry of Research. Ulisse Stefanelli acknowledges partial support from the grants FP7-IDEAS-ERC-StG \#200497 (BioSMA), CNR-AV\v CR 2010-2012 (SmartMath), and the Alexander von Humboldt Foundation.

{}

\bigskip
{\small\noindent
Dorin Bucur:
Laboratoire de Math\'ematiques (LAMA),
Universit\'e de Savoie\\
Campus Scientifique,
73376 Le-Bourget-Du-Lac - FRANCE\\
{\tt dorin.bucur@univ-savoie.fr}\\
{\tt http://www.lama.univ-savoie.fr/$\sim$bucur/}

\bigskip\noindent
Giuseppe Buttazzo:
Dipartimento di Matematica,
Universit\`a di Pisa\\
Largo B. Pontecorvo 5,
56127 Pisa - ITALY\\
{\tt buttazzo@dm.unipi.it}\\
{\tt http://www.dm.unipi.it/pages/buttazzo/}

\bigskip\noindent
Ulisse Stefanelli:
Istituto di Matematica Applicata e Tecnologie Informatiche,
CNR\\
Via Ferrata 1,
I-27100 Pavia - ITALY\\
and Weierstrass Institute for Applied Analysis and Stochastics,\\
Mohrenstrasse 39, D-10117 Berlin - GERMANY\\
{\tt ulisse.stefanelli@imati.cnr.it}\\
{\tt http://www.imati.cnr.it/ulisse}}
\end{document}